\documentclass[12pt,oneside,english]{amsart}
\usepackage[T1]{fontenc}
\usepackage[latin1]{inputenc}
\usepackage{geometry}
\usepackage{setspace}
\usepackage{amssymb}

\makeatletter
 \theoremstyle{plain}
\newtheorem{thm}{Theorem}[section]
  \theoremstyle{plain}
  \newtheorem{lem}[thm]{Lemma}
  \theoremstyle{plain}
  \newtheorem{prop}[thm]{Proposition}
  \theoremstyle{definition}
  \newtheorem{defn}[thm]{Definition}
  \theoremstyle{plain}
  \newtheorem{cor}[thm]{Corollary}
  \theoremstyle{remark}
  \newtheorem{rem}[thm]{Remark}

\numberwithin{equation}{section}

\usepackage{babel}
\makeatother
\begin{document}

\title{limit theorems for additive c-free convolution}

\author{jiun-chau wang}

\date{May 13, 2008}

\begin{abstract}
In this paper we determine the limiting distributional behavior for
sums of infinitesimal c-free random variables. We show that the weak
convergence of classical convolution and that of c-free convolution
are equivalent for measures in an infinitesimal triangular array,
where the measures may have unbounded support. Moreover, we use these
limit theorems to study the c-free infinite divisibility and stability.
These results are obtained by complex analytic methods without reference
to the combinatorics of c-free convolution.
\end{abstract}

\address{Department of Mathematics, Indiana University, Bloomington, Indiana
47405.}

\email{freeprobability@gmail.com}

\subjclass[2000]{Primary: 46L53; Secondary: 60F05.}

\keywords{Additive c-free convolution; Limit theorems; Infinitesimal arrays. }

\maketitle

\section{Introduction}

The theory of the conditionally free (abbreviated as c-free) random
variables was introduced by Bo\.{z}ejko, Leinert and Speicher in
\cite{BLS}, as a generalization of Voiculescu's freeness to the algebras
with two states. The concept of c-freeness leads to a binary operation,
called additive c-free convolution, on pairs of compactly supported
probability measures on the real line. The c-free analogues of central
and Poisson limit theorems for identically distributed summands were
also proved in \cite{BLS}. The development of the c-free probability
theory relies heavily on the combinatorics of non-crossing partitions.
The nature of the combinatorial tools makes it difficult to discuss
limit theorems when the measures do not have finite moments. Even
for finite moments the limit theorems proved in \cite{BLS} and \cite{BB}
require subtle combinatorics arguments.

The aim of this paper is to provide an analytic approach to study
the asymptotic distributional behavior of additive c-free convolution.
As shown in \cite{Mulcfree}, the same approach also works in the
multiplicative context. The extension of (additive) c-free convolution
to measures with unbounded support was done by Belinschi \cite{SerbanNotes}.
His work provided useful inspirations for some of the analytic questions
in our approach, as will be seen below.

The remainder of this paper is organized as follows. In Section 2
we deal with the analytic problems involved in using an analogue of
Voiculescu's $R$-transform for measures without bounded support,
and we extend the definition of c-free convolution to pairs of arbitrary
measures using this transform. Section 3 contains the main result
of this paper (Theorem 3.5), which provides necessary and sufficient
conditions for the weak convergence of c-free convolution of measures
in an infinitesimal array. In Section 4 we present various characterizations
of c-free infinite divisibility, which extend the results in \cite{Krystek}
for pairs of compactly supported measures. Section 5 contains a brief
discussion of c-free stability.

\section{Setting and Basic Properties}

In this section we focus on the analytic apparatus needed for the
calculation of c-free convolution. Most of the results we quoted from
the literature were developed for studying the free and boolean convolutions.
We refer the reader to the book \cite{VDN} for a comprehensive introduction
to free probability theory, and to the papers \cite{SW,BerBoolean}
for a detailed treatment of boolean probability theory.

\subsection{Cauchy transforms and c-free convolution }

Denote by $\mathcal{M}$ the family of all Borel probability measures
on the real line $\mathbb{R}$ and set $\mathbb{C}^{+}=\{ z\in\mathbb{C}:\,\Im z>0\}$,
$\mathbb{C}^{-}=-\mathbb{C}^{+}$. We associate each measure $\mu\in\mathcal{M}$
its \emph{Cauchy transform} \[
G_{\mu}(z)=\int_{-\infty}^{\infty}\frac{1}{z-t}\, d\mu(t),\qquad z\in\mathbb{C}^{+},\]
and its reciprocal $F_{\mu}=1/G_{\mu}:\,\mathbb{C}^{+}\rightarrow\mathbb{C}^{+}$.
The measure $\mu$ can be recovered from $G_{\mu}$ as the $\text{weak}^{*}$-limit
of the measures\[
d\nu_{y}(x)=-\frac{1}{\pi}\Im G_{\mu}(x+iy)\, dx\]
as $y\rightarrow0^{+}$. For $\alpha,\beta>0$, we define the cone
$\Gamma_{\alpha}=\{ x+iy\in\mathbb{C}^{+}:\,\left|x\right|<\alpha y\}$
and the truncated cone $\Gamma_{\alpha,\beta}=\{ x+iy\in\Gamma_{\alpha}:\, y>\beta\}$.
As shown in \cite{BVunbdd}, we have $\Im z\leq\Im F_{\mu}(z)$ for
$z\in\mathbb{C}^{+}$ and \begin{equation}
F_{\mu}(z)=z(1+o(1)),\qquad z\in\mathbb{C}^{+},\label{eq:2.1}\end{equation}
as $z\rightarrow\infty$ \emph{nontangentially} (i.e., $\left|z\right|\rightarrow\infty$
but $z$ stays within a cone $\Gamma_{\alpha}$ for some $\alpha>0$.)
The measure $\mu$ is uniquely determined by the function $F_{\mu}$,
and conversely, any analytic function $F:\,\mathbb{C}^{+}\rightarrow\mathbb{C}^{+}$
so that $F(z)=z(1+o(1))$ as $z\rightarrow\infty$ nontangentially
is of the form $F_{\mu}$ for a unique probability measure $\mu$
on $\mathbb{R}$.

Property (2.1) also implies that, for every $\alpha>0$, there exists
$\beta=\beta(\mu,\alpha)>0$ such that the function $F_{\mu}$ has
a left inverse $F_{\mu}^{-1}$ (relative to composition) defined in
$\Gamma_{\alpha,\beta}$. Moreover, we see that $F_{\mu}^{-1}(z)=z(1+o(1))$
as $z\rightarrow\infty$ nontangentially. For $\mu,\nu\in\mathcal{M}$,
the \emph{additive free convolution} $\mu\boxplus\nu\in\mathcal{M}$
is characterized \cite{BVunbdd} by the identity \[
F_{\mu\boxplus\nu}^{-1}(z)+z=F_{\mu}^{-1}(z)+F_{\nu}^{-1}(z),\]
where $z$ is in a truncated cone $\Gamma_{\alpha,\beta}$ contained
in the domain of all involved functions.

For a measure $\mu\in\mathcal{M}$, observe that the function $E_{\mu}(z)=z-F_{\mu}(z)$
takes values in $\mathbb{C}^{-}\cup\mathbb{R}$ and $E_{\mu}(z)=o(\left|z\right|)$
as $z\rightarrow\infty$ nontangentially. Conversely, any analytic
function $E:\,\mathbb{C}^{+}\rightarrow\mathbb{C}^{-}\cup\mathbb{R}$
with these properties is of the form $E_{\mu}$ for a unique probability
measure $\mu$. The \emph{additive boolean convolution} $\mu\uplus\nu\in\mathcal{M}$
of two measures $\mu,\nu\in\mathcal{M}$ is characterized \cite{SW,BerBoolean}
by \[
E_{\mu\uplus\nu}(z)=E_{\mu}(z)+E_{\nu}(z),\qquad z\in\mathbb{C}^{+}.\]

The theory of c-free convolution for pairs of compactly supported
probability measures was first studied in \cite{BLS}. The c-free
convolution $(\mu_{1},\nu_{1})\boxplus_{\text{c}}(\mu_{2},\nu_{2})$
of such pairs is again a pair of compactly supported probability measures
$(\widetilde{\mu},\widetilde{\nu})$, where the measure $\widetilde{\nu}=\nu_{1}\boxplus\nu_{2}$.
In order to describe the measure $\widetilde{\mu}$, these authors
further introduced, for a pair of compactly supported measures $(\mu,\nu)$,
the analytic function \[
C_{(\mu,\nu)}(z)=z\left[E_{\mu}\left(G_{\nu}^{-1}(z)\right)\right],\]
where the inversion of $G_{\nu}$ is carried out in a neighborhood
of $\infty$, and they proved that \[
C_{(\widetilde{\mu},\widetilde{\nu})}(z)=C_{(\mu_{1},\nu_{1})}(z)+C_{(\mu_{2},\nu_{2})}(z).\]

The starting point for the treatment of measures with unbounded support
is observing that, for arbitrary measures $\mu,\nu\in\mathcal{M}$,
the function $C_{(\mu,\nu)}$ is actually defined in an appropriate
domain. For measures $\mu,\nu\in\mathcal{M}$, we introduce a new
function \begin{equation}
\Phi_{(\mu,\nu)}(z)=E_{\mu}\left(F_{\nu}^{-1}(z)\right)\label{eq:2.2}\end{equation}
in a truncated cone $\Gamma_{\alpha,\beta}$ where the function $F_{\nu}^{-1}$
is defined. The function $\Phi_{(\mu,\nu)}$ is obtained from the
function $C_{(\mu,\nu)}(z)/z$ by a change of variable $z\mapsto1/z$,
and is more suitable for our purposes. It is easy to verify that we
have \[
\Phi_{(\widetilde{\mu},\widetilde{\nu})}(z)=\Phi_{(\mu_{1},\nu_{1})}(z)+\Phi_{(\mu_{2},\nu_{2})}(z)\]
in the case of compactly supported measures.

We will require the following result from \cite{BPstable}, whose
proof is based on the Cauchy integral formula.

\begin{lem}
Let $\alpha,\beta,\varepsilon$ be positive numbers, and let $\phi:\,\Gamma_{\alpha,\beta}\rightarrow\mathbb{C}$
be an analytic function such that $\left|\phi(z)\right|\leq\varepsilon\left|z\right|$
for every $z\in\Gamma_{\alpha,\beta}$. Then, for every $\alpha^{\prime}<\alpha$
and $\beta^{\prime}>\beta$, there exists $K>0$ such that the derivative
$\phi^{\prime}(z)$ is estimated as follows \[
\left|\phi^{\prime}(z)\right|\leq K\varepsilon,\qquad z\in\Gamma_{\alpha^{\prime},\beta^{\prime}}.\]

\end{lem}
The following result was first noted in \cite{SerbanNotes}.

\begin{prop}
Let $\mu_{1},\mu_{2},\nu_{1},\nu_{2}\in\mathcal{M}$, and let $\nu=\nu_{1}\boxplus\nu_{2}$.
Suppose that both $F_{\nu_{1}}^{-1}$ and $F_{\nu_{2}}^{-1}$ are
defined in a cone $\Gamma_{\alpha,\beta}$. Then there exists another
truncated cone $\Gamma_{\alpha^{\prime},\beta^{\prime}}\subset\Gamma_{\alpha,\beta}$
such that the function \[
\Phi(z)=\Phi_{(\mu_{1},\nu_{1})}(z)+\Phi_{(\mu_{2},\nu_{2})}(z),\qquad z\in\Gamma_{\alpha^{\prime},\beta^{\prime}},\]
is of the form $\Phi_{(\mu,\nu)}$ for a unique probability measure
$\mu$ on $\mathbb{R}$.
\end{prop}
\begin{proof}
Note that (2.1) shows that $F_{\nu}(z)\in\Gamma_{\alpha,\beta}$ as
$z\rightarrow\infty$ nontangentially. To prove the proposition, it
suffices to show that the function $E(z)=\Phi\left(F_{\nu}(z)\right)$
is of the form $E_{\mu}(z)$ for a unique probability measure $\mu\in\mathcal{M}$,
that is, to show that the function $E(z)$ extends analytically to
$\mathbb{C}^{+}$ and $E(z)/z\rightarrow0$ as $z\rightarrow\infty$
nontangentially.

To this purpose, we appeal to a subordination result in \cite{Subordination}
(see also \cite{Biane}) for free convolution $\nu_{1}\boxplus\nu_{2}$,
namely, there exist unique analytic functions $\omega_{1},\omega_{2}:\,\mathbb{C}^{+}\rightarrow\mathbb{C}^{+}$
such that $\omega_{j}(z)=z(1+o(1))$, $j=1,2$, as $z\rightarrow\infty$
nontangentially and $F_{\nu}(z)=F_{\nu_{1}}\left(\omega_{1}(z)\right)=F_{\nu_{2}}\left(\omega_{2}(z)\right)$
for all $z\in\mathbb{C}^{+}$. Then, by (2.2), we have \[
E(z)=E_{\mu_{1}}\left(\omega_{1}(z)\right)+E_{\mu_{2}}\left(\omega_{2}(z)\right)\]
in an open subset of $\mathbb{C}^{+}$, and hence the function $E(z)$
extends analytically to the entire upper half-pane $\mathbb{C}^{+}$.

On the other hand, Lemma 2.1 shows that the derivatives $E_{\mu_{j}}^{\prime}(z)=o(1)$,
$j=1,2$, as $z\rightarrow\infty$ nontangentially. It follows that
there exists $M>\beta$ such that \[
\left|E(z)-E_{\mu_{1}}(z)-E_{\mu_{2}}(z)\right|\leq\left|\omega_{1}(z)-z\right|+\left|\omega_{2}(z)-z\right|,\]
for $z\in\Gamma_{\alpha,M}$, and hence we conclude that $E(z)/z\rightarrow0$
as $z\rightarrow\infty$ nontangentially. Thus the proof is complete.
\end{proof}
Proposition 2.2 allows us to make the following definition which will
be used throughout the rest of this paper.

\begin{defn}
Let $\mu_{1},\mu_{2},\nu_{1},\nu_{2}\in\mathcal{M}$, and let $\nu=\nu_{1}\boxplus\nu_{2}$.
The additive c-free convolution $(\mu_{1},\nu_{2})\boxplus_{\text{c}}(\mu_{2},\nu_{2})$
is the pair $(\mu,\nu)$, where $\mu$ is the unique probability measure
provided by Proposition 2.2.
\end{defn}
We will also use the somewhat abused notation \[
\mu=\mu_{1}\boxplus_{\text{c}}\mu_{2}.\]
Indeed, $\mu_{1}\boxplus_{\text{c}}\mu_{2}$ depends on $\nu_{1}$
and $\nu_{2}$ as well. We choose this shorter notation because the
asymptotic behavior of free convolution $\boxplus$ is well understood
(see \cite{CG1}, and \cite{BJadd} for a different approach), and
we would like to address convergence issues on the first component
of c-free convolution. Our second remark is that the operation $\boxplus_{\text{c}}$
is commutative and associative by Proposition 2.2, and it reduces
to the original c-free convolution introduced in \cite{BLS} in the
case of compactly supported measures.

\subsection{Weak convergence of probability measures}

If $\mu_{n}$ and $\mu$ are elements of $\mathcal{M}$, or more generally,
finite Borel measures on $\mathbb{R}$, we say that $\mu_{n}$ converges
\emph{weakly} to $\mu$ if \[
\lim_{n\rightarrow\infty}\int_{-\infty}^{\infty}f(t)\, d\mu_{n}(t)=\lim_{n\rightarrow\infty}\int_{-\infty}^{\infty}f(t)\, d\mu(t)\]
for every bounded continuous function $f$ on $\mathbb{R}$. The weak
convergence of measures requires tightness. Recall that a family $\mathcal{F}$
of finite Borel measures on $\mathbb{R}$ is \emph{tight} if \[
\lim_{y\rightarrow+\infty}\sup_{\mu\in\mathcal{F}}\mu(\{ t:\,\left|t\right|>y\})=0.\]
Any tight sequence of probability measures has a subsequence which
converges weakly to a probability measure.

We note for further reference that weak convergence of probability
measures can be translated in terms of convergence properties of the
corresponding functions $E$ and $\Phi$.

\begin{prop}
Let $\{\mu_{n}\}_{n=1}^{\infty}$ and $\{\nu_{n}\}_{n=1}^{\infty}$
be two sequences in $\mathcal{M}$.
\begin{enumerate}
\item The sequence $\mu_{n}$ converges weakly to a measure $\mu\in\mathcal{M}$
if and only if there exists a truncated cone $\Gamma$ such that the
sequence $E_{\mu_{n}}$ converges uniformly on the compact subsets
of $\Gamma$ to a function $E$, and $E_{\mu_{n}}(z)=o(\left|z\right|)$
uniformly in $n$ as $\left|z\right|\rightarrow\infty$, $z\in\Gamma$.
Moreover, we have $E=E_{\mu}$ in this situation.
\item Assume that the sequence $\nu_{n}$ converges weakly to a measure
$\nu\in\mathcal{M}$. Then the sequence $\mu_{n}$ converges weakly
to a measure $\mu\in\mathcal{M}$ if and only if there exist $\alpha,\beta>0$
such that the functions $\Phi_{(\mu_{n},\nu_{n})}$ are defined in
the cone $\Gamma_{\alpha,\beta}$ for every $n$, $\lim_{n\rightarrow\infty}\Phi_{(\mu_{n},\nu_{n})}(iy)$
exists for every $y>\beta$ and $\Phi_{(\mu_{n},\nu_{n})}(iy)=o(y)$
uniformly in $n$ as $y\rightarrow\infty$. Moreover, in this case
we have $\lim_{n\rightarrow\infty}\Phi_{(\mu_{n},\nu_{n})}(iy)=\Phi_{(\mu,\nu)}(iy)$
for every $y>\beta$.
\end{enumerate}
\end{prop}
\begin{proof}
We refer to \cite{BPstable} for the proof of (1). To prove (2), note
first that the existence of the truncated cone $\Gamma_{\alpha,\beta}$
is provided by the weak convergence of the sequence $\{\nu_{n}\}_{n=1}^{\infty}$
(see \cite[Proposition 2.3]{BPstable}). Moreover, the sequence $F_{\nu_{n}}^{-1}$
converges uniformly on the compact subsets of $\Gamma_{\alpha,\beta}$
to the function $F_{\nu}^{-1}$, and $F_{\nu_{n}}^{-1}(z)=z(1+o(1))$
uniformly in $n$ as $z\rightarrow\infty$, $z\in\Gamma_{\alpha,\beta}$.

Assume that the measures $\mu_{n}$ converge weakly to a measure $\mu$.
Then (1) and Lemma 2.1 imply that the derivatives $E_{\mu}^{\prime}(z)=o(1)$
and $E_{\mu_{n}}^{\prime}(z)=o(1)$ uniformly in $n$ as $z\rightarrow\infty$
nontangentially. It follows that there exists $M>\beta$ such that\begin{eqnarray*}
\left|\Phi_{(\mu_{n},\nu_{n})}(z)-\Phi_{(\mu,\nu)}(z)\right| & = & \left|E_{\mu_{n}}\left(F_{\nu_{n}}^{-1}(z)\right)-E_{\mu}\left(F_{\nu}^{-1}(z)\right)\right|\\
 & \leq & \left|E_{\mu_{n}}\left(F_{\nu_{n}}^{-1}(z)\right)-E_{\mu_{n}}\left(F_{\nu}^{-1}(z)\right)\right|\\
 &  & +\left|E_{\mu_{n}}\left(F_{\nu}^{-1}(z)\right)-E_{\mu}\left(F_{\nu}^{-1}(z)\right)\right|\\
 & \leq & \left|F_{\nu_{n}}^{-1}(z)-F_{\nu}^{-1}(z)\right|+\left|E_{\mu_{n}}\left(F_{\nu}^{-1}(z)\right)-E_{\mu}\left(F_{\nu}^{-1}(z)\right)\right|\end{eqnarray*}
for every $n\in\mathbb{N}$ and $z\in\Gamma_{\alpha,M}$. Hence (1)
implies that $\Phi_{(\mu_{n},\nu_{n})}(z)=o(\left|z\right|)$ uniformly
in $n$ as $z\rightarrow\infty$, $z\in\Gamma_{\alpha,\beta}$. The
family $\{\Phi_{(\mu_{n},\nu_{n})}\}_{n=1}^{\infty}$ is normal, and
hence it has subsequences which converge uniformly on the compact
subsets of $\Gamma_{\alpha,\beta}$. Moreover, the above estimate
and (1) actually imply that the limit of such a subsequence must be
the function $\Phi_{(\mu,\nu)}$. Therefore we conclude that the entire
sequence $\{\Phi_{(\mu_{n},\nu_{n})}\}_{n=1}^{\infty}$ converges
uniformly on the compact subsets of $\Gamma_{\alpha,\beta}$ to the
function $\Phi_{(\mu,\nu)}$. In particular, these results hold for
$z=iy$, $y>\beta$.

Conversely, let us assume that $\lim_{n\rightarrow\infty}\Phi_{(\mu_{n},\nu_{n})}(iy)$
exists for every $y>\beta$ and $\Phi_{(\mu_{n},\nu_{n})}(iy)=o(y)$
uniformly in $n$ as $y\rightarrow\infty$. We first show that the
sequence $\{\mu_{n}\}_{n=1}^{\infty}$ is tight. Let us define $u_{n}=u_{n}(y)=F_{\nu_{n}}^{-1}(iy)=iy+\phi_{\nu_{n}}(iy)$
for $y>\beta$, and also observe that $\phi_{\nu_{n}}(iy)=o(y)$ uniformly
in $n$ as $y\rightarrow\infty$ by the assumption on the weak convergence
of $\{\nu_{n}\}_{n=1}^{\infty}$. Then we have \[
u_{n}-F_{\mu_{n}}(u_{n})=E_{\mu_{n}}(u_{n})=\Phi_{(\mu_{n},\nu_{n})}(iy)=o(y)\]
uniformly in $n$ as $y\rightarrow\infty$. Moreover, note that \[
\left|G_{\mu_{n}}(u_{n}(y))\right|\leq\frac{1}{\Im u_{n}}=\frac{1}{y+o(y)}\]
uniformly in $n$ as $y\rightarrow\infty$. Hence, we conclude that
$u_{n}^{2}G_{\mu_{n}}(u_{n})-u_{n}=o(y)$ uniformly in $n$ as $y\rightarrow\infty$.
On the other hand, since $u_{n}=iy+o(y)$ uniformly in $n$ as $y\rightarrow\infty$,
there exists $M>\beta$ such that \[
\frac{t^{2}}{(\Re u_{n}(y)-t)^{2}+(\Im u_{n}(y))^{2}}\geq\frac{1}{8},\qquad t\in\mathbb{R},\:\left|t\right|\geq y>M,\]
for every $n$. Finally, putting everything together, we have\begin{eqnarray*}
-\frac{1}{y}\Im\left(u_{n}^{2}G_{\mu_{n}}(u_{n})-u_{n}\right) & = & \frac{\Im u_{n}}{y}\int_{-\infty}^{\infty}\frac{t^{2}}{(\Re u_{n}-t)^{2}+(\Im u_{n})^{2}}\, d\mu_{n}(t)\\
 & \geq & \frac{\Im u_{n}}{y}\int_{\left|t\right|\geq y}\frac{1}{8}\, d\mu_{n}(t)=\frac{\Im u_{n}}{8y}\,\mu_{n}(\{ t:\,\left|t\right|\geq y\}),\end{eqnarray*}
for every $n$ and $y>M$, which implies that $\{\mu_{n}\}_{n=1}^{\infty}$
is tight. If $\mu\in\mathcal{M}$ is a weak cluster point of $\{\mu_{n}\}_{n=1}^{\infty}$,
then the first part of the proof shows that the function $\Phi_{(\mu,\nu)}$
is uniquely determined and hence so is the measure $\mu$. Therefore
the sequence $\mu_{n}$ converges weakly to the measure $\mu$.
\end{proof}
Note that, in case $\nu_{n}=\delta_{0}$, Proposition 2.4 gives the
equivalence between the weak convergence of $\{\mu_{n}\}_{n=1}^{\infty}$
and convergence properties of $\{ E_{\mu_{n}}(iy)\}_{n=1}^{\infty}$.

\subsection{Infinite divisibility }

A pair of probability measures $(\mu,\nu)$ is said to be \emph{$\boxplus_{\text{c}}$-infinitely
divisible} if, for every $n\in\mathbb{N}$, there exist measures $\mu_{n},\nu_{n}\in\mathcal{M}$
such that\[
(\mu,\nu)=\underbrace{(\mu_{n},\nu_{n})\boxplus_{\text{c}}(\mu_{n},\nu_{n})\boxplus_{\text{c}}\cdots\boxplus_{\text{c}}(\mu_{n},\nu_{n})}_{n\,\text{times}},\]
in other words, we have \[
\mu=\underbrace{\mu_{n}\boxplus_{\text{c}}\mu_{n}\boxplus_{\text{c}}\cdots\boxplus_{\text{c}}\mu_{n}}_{n\,\text{times}}\qquad\text{and}\qquad\nu=\underbrace{\nu_{n}\boxplus\nu_{n}\boxplus\cdots\boxplus\nu_{n}}_{n\,\text{times}}.\]
The notion of infinite divisibility related to other convolutions
is defined analogously.

The L\'{e}vy-Hin\v{c}in formula (see \cite{Billingsley}) characterizes
the infinite divisibility relative to classical convolution $*$ of
a probability measure in terms of its Fourier transform. Namely, a
measure $\nu\in\mathcal{M}$ is $*$-infinitely divisible if and only
if there exist $\gamma\in\mathbb{R}$ and a finite positive Borel
measure $\sigma$ on $\mathbb{R}$ such that the Fourier transform
$\widehat{\nu}$ of the measure $\nu$ is given by\begin{equation}
\widehat{\nu}(t)=\exp\left[i\gamma t+\int_{-\infty}^{\infty}\left(e^{itx}-1-\frac{itx}{1+x^{2}}\right)\frac{1+x^{2}}{x^{2}}\, d\sigma(x)\right],\qquad t\in\mathbb{R}.\label{eq:2.3}\end{equation}

The free analogue of L\'{e}vy-Hin\v{c}in formula for a $\boxplus$-infinitely
divisible probability measure was proved in \cite{Vadd,BVunbdd}.
A measure $\nu\in\mathcal{M}$ is $\boxplus$-infinitely divisible
if and only if there exist $\gamma\in\mathbb{R}$ and a finite positive
Borel measure $\sigma$ on $\mathbb{R}$ such that \begin{equation}
F_{\nu}^{-1}(z)=\gamma+z+\int_{-\infty}^{\infty}\frac{1+tz}{z-t}\, d\sigma(t),\qquad z\in\mathbb{C}^{+}.\label{eq:2.4}\end{equation}
In other words, the function $F_{\nu}^{-1}$ can be extended analytically
to $\mathbb{C}^{+}$ if the measure $\nu$ is $\boxplus$-infinitely
divisible.

Every measure $\nu\in\mathcal{M}$ is $\uplus$-infinitely divisible
\cite{SW}. The reason for this is that every analytic self-mapping
of $\mathbb{C}^{+}$ has a Nevanlinna integral representation \cite{Achieser}.
In particular, the function $E_{\nu}$ can be written as \begin{equation}
E_{\nu}(z)=\gamma+\int_{-\infty}^{\infty}\frac{1+tz}{z-t}\, d\sigma(t),\qquad z\in\mathbb{C}^{+},\label{eq:2.5}\end{equation}
where $\gamma\in\mathbb{R}$ and $\sigma$ is a finite positive Borel
measure on $\mathbb{R}$.

In the sequel, we will use the notations $\nu_{*}^{\gamma,\sigma}$,
$\nu_{\boxplus}^{\gamma,\sigma}$ and $\nu_{\uplus}^{\gamma,\sigma}$
to denote respectively the $*$-, $\boxplus$-, and $\uplus$-infinitely
divisible measures that are uniquely determined by $\gamma$ and $\sigma$
via the formulas (2.3), (2.4) and (2.5).

\section{Proof of the Main Result}

Let $\{ k_{n}\}_{n=1}^{\infty}$ be a sequence of positive integers,
and let $\{ c_{n}\}_{n=1}^{\infty}$ and $\{ c_{n}^{\prime}\}_{n=1}^{\infty}$
be two sequences in $\mathbb{R}$. Consider two \emph{infinitesimal}
triangular arrays $\{\mu_{nk}:\, n\in\mathbb{N},1\leq k\leq k_{n}\}$
and $\{\nu_{nk}:\, n\in\mathbb{N},1\leq k\leq k_{n}\}$ in $\mathcal{M}$.
Here the infinitesimality of the array $\{\mu_{nk}\}_{n,k}$ means
that \[
\lim_{n\rightarrow\infty}\max_{1\leq k\leq k_{n}}\mu_{nk}(\{ t\in\mathbb{R}:\,\left|t\right|\geq\varepsilon\})=0,\]
for every $\varepsilon>0$. The goal of this section is to study the
asymptotic behavior of the sequence $\{(\mu_{n},\nu_{n})\}_{n=1}^{\infty}$,
where \[
(\mu_{n},\nu_{n})=(\delta_{c_{n}},\delta_{c_{n}^{\prime}})\boxplus_{\text{c}}(\mu_{n1},\nu_{n1})\boxplus_{\text{c}}(\mu_{n2},\nu_{n2})\boxplus_{\text{c}}\cdots\boxplus_{\text{c}}(\mu_{nk_{n}},\nu_{nk_{n}}),\]
and $\delta_{c}$ denotes the Dirac point mass at $c\in\mathbb{R}$.

To this purpose, we introduce the measures $\mu_{nk}^{\circ}$ by
setting \[
d\mu_{nk}^{\circ}(t)=d\mu_{nk}(t+a_{nk}),\]
where the numbers $a_{nk}\in[-1,1]$ are given by \begin{equation}
a_{nk}=\int_{\left|t\right|<1}t\, d\mu_{nk}(t).\label{eq:3.1}\end{equation}
Note that the array $\{\mu_{nk}^{\circ}\}_{n,k}$ is infinitesimal
and $\lim_{n\rightarrow\infty}\max_{1\leq k\leq k_{n}}\left|a_{nk}\right|=0$.

We also associate each measure $\mu_{nk}^{\circ}$ an analytic function
\[
f_{nk}(z)=\int_{-\infty}^{\infty}\frac{tz}{z-t}\, d\mu_{nk}^{\circ}(t),\qquad z\in\mathbb{C}^{+}.\]
Observe that $\Im f_{nk}(z)<0$ for all $z\in\mathbb{C}^{+}$ unless
the measure $\mu_{nk}^{\circ}=\delta_{0}$, and that $f_{nk}(z)=o(\left|z\right|)$
as $z\rightarrow\infty$ nontangentially.

We will require the following result.

\begin{prop}
Let $\Gamma_{\alpha,\beta}$ be a truncated cone, and let $\{ c_{n}\}_{n=1}^{\infty}$
be a sequence in $\mathbb{R}$. Suppose that the arrays $\{\mu_{nk}\}_{n,k}$
and $\{\nu_{nk}\}_{n,k}$ in $\mathcal{M}$ are infinitesimal, and
that the centered measures $\mu_{nk}^{\circ}$ are defined as above.
Then
\begin{enumerate}
\item $E_{\mu_{nk}^{\circ}}(z)=f_{nk}(z+a_{nk})(1+v_{nk}(z))$ for sufficiently
large $n$, where the sequence $v_{n}(z)=\max_{1\leq k\leq k_{n}}\left|v_{nk}(z)\right|$
has the properties that $\lim_{n\rightarrow\infty}v_{n}(z)=0$ for
all $z\in\Gamma_{\alpha,\beta}$ and $v_{n}(z)=o(1)$ uniformly in
$n$ as $\left|z\right|\rightarrow\infty$, $z\in\Gamma_{\alpha,\beta}$.
\item For every $n$, $k$ and $z,\, w\in\Gamma_{\alpha,\beta}$, we have
\[
\left|f_{nk}(w)-f_{nk}(z)\right|\leq\left|f_{nk}(z)\right|\frac{\left|z-w\right|}{\Im z}\left(1+\sqrt{1+\alpha^{2}}\left|\frac{z}{w}-1\right|\right).\]

\item For every $y>\beta$, the sequence $\{ c_{n}+\sum_{k=1}^{k_{n}}E_{\mu_{nk}}(iy)\}_{n=1}^{\infty}$
converges if and only if the sequence $\{ c_{n}+\sum_{k=1}^{k_{n}}\left[a_{nk}+f_{nk}(iy)\right]\}_{n=1}^{\infty}$
converges. Moreover, the two sequences have the same limit.
\item If \[
L=\sup_{n\geq1}\sum_{k=1}^{k_{n}}\int_{-\infty}^{\infty}\frac{t^{2}}{1+t^{2}}\, d\mu_{nk}^{\circ}(t)<+\infty,\]
then $c_{n}+\sum_{k=1}^{k_{n}}E_{\mu_{nk}}(iy)=o(y)$ uniformly in
$n$ as $y\rightarrow\infty$ if and only if $c_{n}+\sum_{k=1}^{k_{n}}\left[a_{nk}+f_{nk}(iy)\right]=o(y)$
uniformly in $n$ as $y\rightarrow\infty$.
\end{enumerate}
\end{prop}
\begin{proof}
(1), (3) and (4) are proved in \cite{JCBoolean}. To prove (2), let
us consider the analytic function\[
f_{\mu}(z)=\int_{-\infty}^{\infty}\frac{tz}{z-t}\, d\mu(t),\qquad z\in\mathbb{C}^{+},\]
for a measure $\mu\in\mathcal{M}$. For $z,w\in\mathbb{C}^{+}$, we
have\[
\left|f_{\mu}(z)-f_{\mu}(w)\right|\leq\left|z-w\right|\int_{-\infty}^{\infty}\frac{t^{2}}{\left|w-t\right|\left|z-t\right|}\, d\mu(t)\]
and\[
\Im z\,\int_{-\infty}^{\infty}\frac{t^{2}}{\left|z-t\right|^{2}}\, d\mu(t)=\left|\Im f_{\mu}(z)\right|\leq\left|f_{\mu}(z)\right|.\]
In addition, we have \begin{eqnarray*}
\left|\frac{z-t}{w-t}\right| & \leq & \frac{\left|z-w\right|+\left|w-t\right|}{\left|w-t\right|}\\
 & = & 1+\left|\frac{w}{w-t}\right|\left|\frac{z}{w}-1\right|\\
 & \leq & 1+\sqrt{1+\alpha^{2}}\left|\frac{z}{w}-1\right|\end{eqnarray*}
for every $t\in\mathbb{R}$ and $z,w\in\Gamma_{\alpha}$. Therefore
(2) follows from these considerations.
\end{proof}
It was first observed in \cite{BPHincin} that for any given truncated
cone $\Gamma_{\alpha,\beta}$, the function $F_{\mu}^{-1}$ is defined
in $\Gamma_{\alpha,\beta}$ as long as the measure $\mu$ concentrates
near the origin. More precisely, for given $\alpha,\beta>0$, there
exists $\varepsilon>0$ with the property that if $\mu\in\mathcal{M}$
is such that $\mu(\{ t\in\mathbb{R}:\,\left|t\right|\geq\varepsilon\})<\varepsilon$,
then the function $F_{\mu}^{-1}$ is defined in $\Gamma_{\alpha,\beta}$.

\begin{lem}
Let $\Gamma_{\alpha,\beta}$ be a truncated cone, and let $\{\mu_{nk}\}_{n,k}$
and $\{\nu_{nk}\}_{n,k}$ be two infinitesimal arrays in $\mathcal{M}$.
Then, for sufficiently large $n$, we have \[
\Phi_{(\mu_{nk},\nu_{nk})}(z)-a_{nk}=f_{nk}(z)(1+u_{nk}(z)),\qquad z\in\Gamma_{\alpha,\beta},\:1\leq k\leq k_{n},\]
where the sequence \[
u_{n}(z)=\max_{1\leq k\leq k_{n}}\left|u_{nk}(z)\right|\]
has the properties that $\lim_{n\rightarrow\infty}u_{n}(z)=0$ for
all $z\in\Gamma_{\alpha,\beta}$, and that $u_{n}(z)=o(1)$ uniformly
in $n$ as $\left|z\right|\rightarrow\infty$, $z\in\Gamma_{\alpha,\beta}$.
\end{lem}
\begin{proof}
Introduce measures\[
d\nu_{nk}^{\circ}(t)=d\nu_{nk}(t+a_{nk}),\]
where the real numbers $a_{nk}$ are defined as in (3.1). The infinitesimality
of the arrays $\{\nu_{nk}\}_{n,k}$ and $\{\nu_{nk}^{\circ}\}_{n,k}$
and the remark we make prior to the current lemma imply, as $n$ tends
to infinity, that the functions $F_{\nu_{nk}}^{-1}$ and $F_{\nu_{nk}^{\circ}}^{-1}$
are defined in the cone $\Gamma_{\alpha,\beta}$ and moreover $F_{\nu_{nk}}^{-1}(z)=z(1+o(1))$
uniformly in $k$ and $z\in\Gamma_{\alpha,\beta}$.

The desired result now follows from (1) and (2) of Proposition 3.1,
and from the following observation:\[
\Phi_{(\mu_{nk},\nu_{nk})}(z)-a_{nk}=\Phi_{(\mu_{nk}^{\circ},\nu_{nk}^{\circ})}(z)=E_{\mu_{nk}^{\circ}}\left(F_{\nu_{nk}^{\circ}}^{-1}(z)\right)=E_{\mu_{nk}^{\circ}}\left(F_{\nu_{nk}}^{-1}(z)-a_{nk}\right).\]

\end{proof}
As shown in \cite{BJadd}, the real and the imaginary parts of the
function $f_{nk}$ become comparable when $n$ is large. More precisely,
we have \[
\left|\Re f_{nk}(iy)\right|\leq(3+6y)\left|\Im f_{nk}(iy)\right|,\qquad1\leq k\leq k_{n},\, y\geq1,\]
and \[
\left|\Re f_{nk}(iy)\right|\leq2\left|\Im f_{nk}(iy)\right|+\left|b_{nk}(y)\right|,\qquad1\leq k\leq k_{n},\, y\geq1,\]
where $n$ is sufficiently large and the real-valued function $b_{nk}(y)$
is defined by \[
b_{nk}(y)=\int_{\left|t\right|\geq1}\left[a_{nk}+\frac{(t-a_{nk})y^{2}}{y^{2}+(t-a_{nk})^{2}}\right]\, d\mu_{nk}(t).\]

We will need an auxiliary result from \cite{JCBoolean}, where it
was written in a slightly different form.

\begin{lem}
Consider a triangular array $\{ s_{nk}\}_{n,k}$ in $[0,+\infty)$
and two arrays $\{ z_{nk}\}_{n,k}$ , $\{ w_{nk}\}_{n,k}$ in $\mathbb{C}$.
Let $\{ c_{n}\}_{n=1}^{\infty}$ be a sequence in $\mathbb{R}$. Assume
that
\begin{enumerate}
\item $\Im w_{nk}\leq0$ and $\Im z_{nk}\leq0$ for all $n$ and $k$;
\item $z_{nk}=w_{nk}(1+\varepsilon_{nk})$ and $\lim_{n\rightarrow\infty}\varepsilon_{n}=0$,
where $\varepsilon_{n}=\max_{1\leq k\leq k_{n}}\left|\varepsilon_{nk}\right|$;
\item there exists a constant $M>0$ such that $\left|\Re w_{nk}\right|\leq M\left|\Im w_{nk}\right|+s_{nk}$
for all $n$ and $k$.
\end{enumerate}
Then, for sufficiently large $n$, we have \[
\left|\sum_{k=1}^{k_{n}}\left[z_{nk}-w_{nk}\right]\right|\leq(1+M)\varepsilon_{n}\left|\sum_{k=1}^{k_{n}}\Im w_{nk}\right|+\varepsilon_{n}\sum_{k=1}^{k_{n}}s_{nk},\]
and \[
(1-\varepsilon_{n}-\varepsilon_{n}M)\left|\sum_{k=1}^{k_{n}}\Im w_{nk}\right|\leq\left|\sum_{k=1}^{k_{n}}\Im z_{nk}\right|+\varepsilon_{n}\sum_{k=1}^{k_{n}}s_{nk}.\]
In particular, if $\sup_{n\geq1}\sum_{k=1}^{k_{n}}s_{nk}<+\infty$,
then the sequence $\{ c_{n}+\sum_{k=1}^{k_{n}}z_{nk}\}_{n=1}^{\infty}$
converges if and only the sequence $\{ c_{n}+\sum_{k=1}^{k_{n}}w_{nk}\}_{n=1}^{\infty}$
does. Moreover, the two sequences have the same limit.

\end{lem}
\begin{prop}
Let $\{\mu_{nk}\}_{n,k}$ and $\{\nu_{nk}\}_{n,k}$ be two infinitesimal
arrays in $\mathcal{M}$, and let $\{ c_{n}\}_{n=1}^{\infty}$ be
a sequence of real numbers. Given $\beta\geq1$, suppose $\Gamma_{\alpha,\beta}$
is the truncated cone where the functions $\Phi_{(\mu_{nk},\nu_{nk})}$
are defined
\begin{enumerate}
\item For every $y>\beta$, the sequence $\{ c_{n}+\sum_{k=1}^{k_{n}}\Phi_{(\mu_{nk},\nu_{nk})}(iy)\}_{n=1}^{\infty}$
converges if and only if the sequence $\{ c_{n}+\sum_{k=1}^{k_{n}}E_{\mu_{nk}}(iy)\}_{n=1}^{\infty}$
does. Moreover, the two sequences have the same limit.
\item If $L<+\infty$ as in \textup{(4)} of \textup{Proposition 3.1}, then
$c_{n}+\sum_{k=1}^{k_{n}}\Phi_{(\mu_{nk},\nu_{nk})}(iy)=o(y)$ uniformly
in $n$ as $y\rightarrow\infty$ if and only if $c_{n}+\sum_{k=1}^{k_{n}}E_{\mu_{nk}}(iy)=o(y)$
uniformly in $n$ as $y\rightarrow\infty$.
\end{enumerate}
\end{prop}
\begin{proof}
It was proved in \cite{BJadd} that $\sum_{k=1}^{k_{n}}\left|b_{nk}(y)\right|\leq5yL$
for sufficiently large $n$ and $y\geq1$. Applying Lemmas 3.2 and
3.3 to arrays $\{ f_{nk}(iy)\}_{n,k}$ and $\{\Phi_{(\mu_{nk},\nu_{nk})}(iy)-a_{nk}\}_{n,k}$,
we conclude that the two sequences $\{ c_{n}+\sum_{k=1}^{k_{n}}\left[a_{nk}+f_{nk}(iy)\right]\}_{n=1}^{\infty}$
and $\{ c_{n}+\sum_{k=1}^{k_{n}}\Phi_{(\mu_{nk},\nu_{nk})}(iy)\}_{n=1}^{\infty}$
have the same asymptotic behavior as $n\rightarrow\infty$. Then the
proof is completed by (3) and (4) of Proposition 3.1.
\end{proof}
We are now ready for the main result of this section. Fix real numbers
$\gamma$, $\gamma^{\prime}$ and finite positive Borel measures $\sigma$,
$\sigma^{\prime}$ on $\mathbb{R}$. Recall that $\nu_{*}^{\gamma,\sigma}$,
$\nu_{\boxplus}^{\gamma,\sigma}$ and $\nu_{\uplus}^{\gamma,\sigma}$
are the $*$-, $\boxplus$-, and $\uplus$-infinitely divisible measures
that we have seen in Section 2.3.

\begin{thm}
Let $\{ c_{n}\}_{n=1}^{\infty}$ and $\{ c_{n}^{\prime}\}_{n=1}^{\infty}$
be two sequences in $\mathbb{R}$, and let $\{\mu_{nk}\}_{n,k}$ and
$\{\nu_{nk}\}_{n,k}$ be two infinitesimal arrays in $\mathcal{M}$.
Suppose that the sequence $\delta_{c_{n}^{\prime}}\boxplus\nu_{n1}\boxplus\nu_{n2}\boxplus\cdots\boxplus\nu_{nk_{n}}$
converges weakly to $\nu_{\boxplus}^{\gamma^{\prime},\sigma^{\prime}}$
as $n\rightarrow\infty$. Then the following assertions are equivalent:
\begin{enumerate}
\item The sequence $\delta_{c_{n}}\boxplus_{\text{c}}\mu_{n1}\boxplus_{\text{c}}\mu_{n2}\boxplus_{\text{c}}\cdots\boxplus_{\text{c}}\mu_{nk_{n}}$
converges weakly to $\mu\in\mathcal{M}$.
\item The sequence $\delta_{c_{n}}\uplus\mu_{n1}\uplus\mu_{n2}\uplus\cdots\uplus\mu_{nk_{n}}$
converges weakly to $\nu_{\uplus}^{\gamma,\sigma}$.
\item The sequence $\delta_{c_{n}}\boxplus\mu_{n1}\boxplus\mu_{n2}\boxplus\cdots\boxplus\mu_{nk_{n}}$
converges weakly to $\nu_{\boxplus}^{\gamma,\sigma}$.
\item The sequence $\delta_{c_{n}}*\mu_{n1}*\mu_{n2}*\cdots*\mu_{nk_{n}}$
converges weakly to $\nu_{*}^{\gamma,\sigma}$.
\item The sequence of measures \[
d\sigma_{n}(t)=\sum_{k=1}^{k_{n}}\frac{t^{2}}{1+t^{2}}\, d\mu_{nk}^{\circ}(t)\]
converges weakly on $\mathbb{R}$ to the measure $\sigma$, and the
sequence of numbers \[
\gamma_{n}=c_{n}+\sum_{k=1}^{k_{n}}\left[a_{nk}+\int_{-\infty}^{\infty}\frac{t}{1+t^{2}}\, d\mu_{nk}^{\circ}(t)\right]\]
converges to $\gamma$ as $n\rightarrow\infty$.
\end{enumerate}
Moreover, if \textup{(1)}-\textup{(5)} are satisfied, then we have
$\Phi_{(\mu,\nu_{\boxplus}^{\gamma^{\prime},\sigma^{\prime}})}=E_{\nu_{\uplus}^{\gamma,\sigma}}$
in a truncated cone.

\end{thm}
\begin{proof}
The equivalence of (2), (3), (4) and (5) was proved in \cite{JCBoolean}.
We will show the equivalence of (1) and (2). Assume that (1) holds.
Define\[
\mu_{n}=\delta_{c_{n}}\boxplus_{\text{c}}\mu_{n1}\boxplus_{\text{c}}\mu_{n2}\boxplus_{\text{c}}\cdots\boxplus_{\text{c}}\mu_{nk_{n}},\;\nu_{n}=\delta_{c_{n}^{\prime}}\boxplus\nu_{n1}\boxplus\nu_{n2}\boxplus\cdots\boxplus\nu_{nk_{n}},\]
and \[
\rho_{n}=\delta_{c_{n}}\uplus\mu_{n1}\uplus\mu_{n2}\uplus\cdots\uplus\mu_{nk_{n}},\qquad n\in\mathbb{N}.\]
Then, by the weak convergence of $\{\nu_{n}\}_{n=1}^{\infty},$ there
exists a truncated cone $\Gamma_{\alpha,\beta}$ such that the functions
$\Phi_{(\mu_{n},\nu_{n})}$ are defined in $\Gamma_{\alpha,\beta}$.
Thus we have \[
\Phi_{(\mu_{n},\nu_{n})}(z)=c_{n}+\sum_{k=1}^{k_{n}}\Phi_{(\mu_{nk},\nu_{nk})}(z)\]
in the cone $\Gamma_{\alpha,\beta}$ and \[
E_{\rho_{n}}(z)=c_{n}+\sum_{k=1}^{k_{n}}E_{\mu_{nk}}(z),\qquad z\in\mathbb{C}^{+}.\]
Also, note that \begin{equation}
c_{n}+\sum_{k=1}^{k_{n}}\left[a_{nk}+f_{nk}(z)\right]=\gamma_{n}+\int_{-\infty}^{\infty}\frac{1+tz}{z-t}\, d\sigma_{n}(t),\label{eq:3.2}\end{equation}
and that the quantity $L$ as in (4) of Proposition 3.1 is precisely
$\sup_{n\geq1}\sigma_{n}(\mathbb{R})$.

Propositions 2.4, 3.1 and 3.4 imply that \[
\lim_{n\rightarrow\infty}E_{\rho_{n}}(iy)=\Phi_{(\mu,\nu_{\boxplus}^{\gamma^{\prime},\sigma^{\prime}})}(iy)=\lim_{n\rightarrow\infty}\left(c_{n}+\sum_{k=1}^{k_{n}}\left[a_{nk}+f_{nk}(iy)\right]\right),\qquad y>\beta.\]
Since $\{ c_{n}+\sum_{k=1}^{k_{n}}\left[a_{nk}+f_{nk}\right]\}_{n=1}^{\infty}$
is a normal family, an application of Montel's theorem shows that
the sequence $\{ c_{n}+\sum_{k=1}^{k_{n}}\left[a_{nk}+f_{nk}(i)\right]\}_{n=1}^{\infty}$
converges to $\Phi_{(\mu,\nu_{\boxplus}^{\gamma^{\prime},\sigma^{\prime}})}(i)$.
Hence (3.2) implies that \begin{eqnarray*}
\lim_{n\rightarrow\infty}\sigma_{n}(\mathbb{R}) & = & \lim_{n\rightarrow\infty}-\Im\left(c_{n}+\sum_{k=1}^{k_{n}}\left[a_{nk}+f_{nk}(i)\right]\right)\\
 & = & -\Im\Phi_{(\mu,\nu_{\boxplus}^{\gamma^{\prime},\sigma^{\prime}})}(i)<+\infty.\end{eqnarray*}
We deduce that $L=\sup_{n\geq1}\sigma_{n}(\mathbb{R})<+\infty$, and
therefore (2) holds by Propositions 2.4 and 3.4. Moreover, in this
case we have $\Phi_{(\mu,\nu_{\boxplus}^{\gamma^{\prime},\sigma^{\prime}})}=E_{\nu_{\uplus}^{\gamma,\sigma}}$
in the cone $\Gamma_{\alpha,\beta}$ by the uniqueness principle in
complex analysis.

Conversely, suppose now (2) holds. Using the equivalence of (2) and
(5), we see that $L<+\infty$ and hence (1) follows again from Propositions
2.4 and 3.4.
\end{proof}
Theorem 3.5 shows that the reciprocal of the Cauchy transform of the
limit law $\mu$ is given by \begin{equation}
F_{\mu}(z)=z-E_{\nu_{\uplus}^{\gamma,\sigma}}\left(F_{\nu_{\boxplus}^{\gamma^{\prime},\sigma^{\prime}}}(z)\right),\qquad z\in\mathbb{C}^{+}.\label{eq:3.3}\end{equation}
Therefore, in order to determine the limit law $\mu$, one first finds
the parameters $\gamma$, $\gamma^{\prime}$, $\sigma$ and $\sigma^{\prime}$
by (5) of Theorem 3.5, then uses the formulas (2.4) and (2.5) to obtain
the function $F_{\mu}$ from (3.3). Finally, the measure $\mu$ is
recovered from the function $G_{\mu}$ as we have seen in Section
2.1.

In this spirit, we see that the results in \cite{BLS} concerning
the c-free analogues of the central and Poisson limit theorems are
direct consequences of Theorem 3.5. Indeed, given $\alpha,\beta\geq0$,
in case $\gamma=\gamma^{\prime}=0$, $\sigma=\alpha^{2}\delta_{0}$
and $\sigma^{\prime}=\beta^{2}\delta_{0}$, the limit law $\mu$ is
a c-free version of the centered Gaussian distribution on $\mathbb{R}$
which appeared in \cite[Theorem 4.3]{BLS}. A c-free analogue of the
Poisson law as in \cite[Theorem 4.4]{BLS} is obtained when $\gamma=\alpha/2$,
$\gamma^{\prime}=\beta/2$, $\sigma=(\alpha/2)\delta_{1}$ and $\sigma^{\prime}=(\beta/2)\delta_{1}$.

It is also interesting to note that (3.3) shows that the limit law
$\mu=\delta_{0}$ if and only if $\gamma=0$ and the measure $\sigma=\delta_{0}$.
Thus, by Theorem 3.5, one obtains necessary and sufficient conditions
for the weak convergence to the Dirac measure at the origin, which
can be viewed as the c-free analogue of the weak law of large numbers.

\section{Application to the $\boxplus_{\text{c}}$-infinite Divisibility}

In this section we give various characterizations of the $\boxplus_{\text{c}}$-infinite
divisibility with the help of Theorem 3.5. The analogue of Theorem
4.1 for compactly supported measures was obtained earlier in \cite{Krystek}
by analyzing the solutions of a complex Burger's equation. The approach
we presented here deals with general probability measures, and does
not involve such a differential equation.

Before outlining the main result we need a definition. A family of
pairs $\{(\mu_{t},\nu_{t})\}_{t\geq0}$ of probability measures on
$\mathbb{R}$ is said to be a \emph{weakly continuous semigroup} relative
to the convolution $\boxplus_{\text{c}}$ if $(\mu_{t},\nu_{t})\boxplus_{\text{c}}(\mu_{s},\nu_{s})=(\mu_{t+s},\nu_{t+s})$
for $t,s\geq0$, and the maps $t\mapsto\mu_{t}$ and $t\mapsto\nu_{t}$
are continuous.

\begin{thm}
Given a $\boxplus$-infinitely divisible measure $\nu\in\mathcal{M}$
and a measure $\mu\in\mathcal{M}$, the following statements are equivalent:
\begin{enumerate}
\item The pair $(\mu,\nu)$ is $\boxplus_{\text{c}}$-infinitely divisible.
\item There exists a real number $\gamma$ and a finite positive Borel measure
$\sigma$ on $\mathbb{R}$ such that the function\[
\Phi_{(\mu,\nu)}(z)=\gamma+\int_{-\infty}^{\infty}\frac{1+tz}{z-t}\, d\sigma(t),\qquad z\in\mathbb{C}^{+}.\]

\item The function $\Phi_{(\mu,\nu)}$ can be analytically continued to
$\mathbb{C}^{+}$.
\item There exists a weakly continuous semigroup $\{(\mu_{t},\nu_{t})\}_{t\geq0}$
relative to $\boxplus_{\text{c}}$ such that $(\mu_{0},\nu_{0})=(\delta_{0},\delta_{0})$
and $(\mu_{1},\nu_{1})=(\mu,\nu)$.
\end{enumerate}
Moreover, if statements \textup{(1)} to \textup{(4)} are all satisfied,
then the limit \[
\gamma=\lim_{t\rightarrow0^{+}}\left[\frac{1}{t}\int_{-\infty}^{\infty}\frac{x}{1+x^{2}}\, d\mu_{t}(x)\right]\]
exists and the measure $\sigma$ is the weak limit of measures \[
\frac{1}{t}\,\frac{x^{2}}{1+x^{2}}\, d\mu_{t}(x)\]
as $t\rightarrow0^{+}$.

\end{thm}
\begin{proof}
We first prove that (1) implies (2). Assume that (1) holds. For every
$n\in\mathbb{N}$, we have \[
\mu=\underbrace{\mu_{n}\boxplus_{\text{c}}\mu_{n}\boxplus_{\text{c}}\cdots\boxplus_{\text{c}}\mu_{n}}_{n\,\text{times}}\qquad\text{and}\qquad\nu=\underbrace{\nu_{n}\boxplus\nu_{n}\boxplus\cdots\boxplus\nu_{n}}_{n\,\text{times}},\]
where $\mu_{n},\nu_{n}\in\mathcal{M}$. Then we have $F_{\nu_{n}}^{-1}(z)-z=\left[F_{\nu}^{-1}(z)-z\right]/n$,
and hence the measures $\nu_{n}$ converge weakly to $\delta_{0}$
as $n\rightarrow\infty$ by Proposition 2.3 in \cite{BPstable}. On
the other hand, the identity $\Phi_{(\mu_{n},\nu_{n})}(z)=\Phi_{(\mu,\nu)}(z)/n$
and Proposition 2.4 imply that the measures $\mu_{n}$ converge weakly
to $\delta_{0}$ as well. Let us introduce two infinitesimal arrays
$\{\mu_{nk}\}_{n,k}$ and $\{\nu_{nk}\}_{n,k}$ by setting $\mu_{nk}=\mu_{n}$
and $\nu_{nk}=\nu_{n}$, where $1\leq k\leq n$. Then the measure
$\mu$ (resp., $\nu$) can be viewed as the weak limit of the c-free
(resp., free) convolutions $\mu_{n1}\boxplus_{\text{c}}\mu_{n2}\boxplus_{\text{c}}\cdots\boxplus_{\text{c}}\mu_{nn}$
(resp., $\nu_{n1}\boxplus\nu_{n2}\boxplus\cdots\boxplus\nu_{nn}$).
Hence (2) follows form Theorem 3.5.

The equivalence of (2) and (3) is based on the Nevanlinna integral
representation of analytic self-mappings in $\mathbb{C}^{+}$(see
\cite{Achieser}).

We next show that (2) implies (4). Suppose that (2) holds. It was
proved in \cite{BVunbdd} that there exists a weakly continuous semigroup
$\{\nu_{t}\}_{t\geq0}$ relative to $\boxplus$ so that $\nu_{0}=\delta_{0}$
and $\nu_{1}=\nu$. Then, for every $t\ge0$, there exists a unique
probability measure $\mu_{t}$ on $\mathbb{R}$ such that $E_{\mu_{t}}(z)=t\left(\Phi_{(\mu,\nu)}\left(F_{\nu_{t}}(z)\right)\right)$
for all $z\in\mathbb{C}^{+}$, where $\mu_{0}=\delta_{0}$. It is
easy to see that the c-free convolution semigoup $\{(\mu_{t},\nu_{t})\}_{t\geq0}$
has the desired properties.

The implication form (4) to (1) is obvious. To finish the proof, we
only need to show the assertions about the measure $\sigma$ and the
number $\gamma$. Assume that the pair $(\mu,\nu)$ is $\boxplus_{\text{c}}$-infinitely
divisible, and let $\{(\mu_{t},\nu_{t})\}_{t\geq0}$ be the corresponding
convolution semigroup as in (4). Let $\{ t_{n}\}_{n=1}^{\infty}$
be a sequence of positive real numbers such that $\lim_{n\rightarrow\infty}t_{n}=0$.
Let $k_{n}=[1/t_{n}]$ for every $n\in\mathbb{N}$, where $[x]$ denotes
the largest integer that is no greater than the real number $x$.
Observe that \[
1-t_{n}<t_{n}k_{n}\leq1,\qquad n\in\mathbb{N}.\]
Hence we have $\lim_{n\rightarrow\infty}t_{n}k_{n}=1$, and further
the properties of the semigroup $\{(\mu_{t},\nu_{t})\}_{t\geq0}$
show that the c-free convolutions\[
\underbrace{\mu_{t_{n}}\boxplus_{\text{c}}\mu_{t_{n}}\boxplus_{\text{c}}\cdots\boxplus_{\text{c}}\mu_{t_{n}}}_{k_{n}\,\text{times}}=\mu_{t_{n}k_{n}}\]
converge weakly to the measure $\mu_{1}=\mu$ as $n\rightarrow\infty$.
Theorem 3.5 then implies that the measures \[
\frac{1}{t_{n}}\,\frac{x^{2}}{1+x^{2}}\, d\mu_{t_{n}}^{\circ}(x)=\frac{1}{t_{n}k_{n}}\, k_{n}\frac{x^{2}}{1+x^{2}}\, d\mu_{t_{n}}^{\circ}(x)\]
converge weakly to the measure $\sigma$ and \[
\gamma=\lim_{n\rightarrow\infty}\left[\frac{1}{t_{n}}\int_{-\infty}^{\infty}\frac{x}{1+x^{2}}\, d\mu_{t}^{\circ}(x)\right],\]
where the centered measures $d\mu_{t_{n}}^{\circ}(x)=d\mu_{t_{n}}(x+a_{n})$
and the numbers $a_{n}$ are defined as in (3.1). The desired result
follows from the facts that $\lim_{n\rightarrow\infty}a_{n}=0$, and
that the topology on the set $\mathcal{M}$ determined by the weak
convergence of measures is actually metrizable \cite[Problem 14.5]{Billingsley}.
\end{proof}
We conclude this section by showing a result, which is a c-free analogue
of Hin\v{c}in's classical theorem on the $*$-infinite divisibility
\cite{Hincin}.

\begin{cor}
Let $\{ c_{n}\}_{n=1}^{\infty}$ and $\{ c_{n}^{\prime}\}_{n=1}^{\infty}$
be two sequences in $\mathbb{R}$, and let $\{\mu_{nk}\}_{n,k}$ and
$\{\nu_{nk}\}_{n,k}$ be two infinitesimal arrays in $\mathcal{M}$.
Suppose that the sequence $\delta_{c_{n}}\boxplus_{\text{c}}\mu_{n1}\boxplus_{\text{c}}\mu_{n2}\boxplus_{\text{c}}\cdots\boxplus_{\text{c}}\mu_{nk_{n}}$
converges weakly to $\mu$, and that the sequence $\delta_{c_{n}^{\prime}}\boxplus\nu_{n1}\boxplus\nu_{n2}\boxplus\cdots\boxplus\nu_{nk_{n}}$
converges weakly to $\nu$. Then the pair $(\mu,\nu)$ is $\boxplus_{\text{c}}$-infinitely
divisible.
\end{cor}
\begin{proof}
It was proved in \cite{BPHincin} that the measure $\nu$ must be
$\boxplus$-infinitely divisible. Therefore the result follows immediately
from Theorems 3.5 and 4.1.
\end{proof}

\section{Stable Laws}

In this section we determine all $\boxplus_{\text{c}}$-stable pairs
of measures, which are defined as follows. Denote by $\mathcal{M}\times\mathcal{M}$
the set of all pairs of measures $(\mu,\nu)$, where $\mu,\nu\in\mathcal{M}$.
Two pairs of measures $(\mu_{1},\nu_{1})$ and $(\mu_{2},\nu_{2})$
in $\mathcal{M}\times\mathcal{M}$ are said to be \emph{equivalent}
if there exist real numbers $a,b$, with $a>0$, such that $d\mu_{2}(t)=d\mu_{1}(at+b)$
and $d\nu_{2}(t)=d\nu_{1}(at+b)$; we indicate this by writing $(\mu_{1},\nu_{1})\sim(\mu_{2},\nu_{2})$.
By analogy with classical probability theory, we say a pair of measures
$(\mu,\nu)\in\mathcal{M}\times\mathcal{M}$ is \emph{$\boxplus_{\text{c}}$-stable}
if $(\mu_{1},\nu_{1})\boxplus_{\text{c}}(\mu_{2},\nu_{2})\sim(\mu,\nu)$
whenever $(\mu_{1},\nu_{1})\sim(\mu,\nu)\sim(\mu_{2},\nu_{2})$.

\begin{rem}
Note that if $d\mu_{2}(t)=d\mu_{1}(at+b)$ and $d\nu_{2}(t)=d\nu_{1}(at+b)$,
where $a>0$, then (2.2) shows that \begin{equation}
\Phi_{(\mu_{2},\nu_{2})}(z)=\frac{1}{a}\left[\Phi_{(\mu_{1},\nu_{1})}(az)-b\right]\label{eq:5.1}\end{equation}
in a truncated cone. Conversely, if pairs $(\mu_{1},\nu_{1})$ and
$(\mu_{2},\nu_{2})$ are such that $d\nu_{2}(t)=d\nu_{1}(at+b)$,
where $a>0$, and (5.1) holds in a truncated cone, then \[
d\mu_{2}(t)=d\mu_{1}(at+b).\]

\end{rem}
\begin{prop}
If $(\mu,\nu)$ is $\boxplus_{\text{c}}$-stable, then $(\mu,\nu)$
is $\boxplus_{\text{c}}$-infinitely divisible.
\end{prop}
\begin{proof}
The $\boxplus_{\text{c}}$-stability of $(\mu,\nu)$ implies that
$(\mu\boxplus_{\text{c}}\mu,\nu\boxplus\nu)=(\mu,\nu)\boxplus_{\text{c}}(\mu,\nu)\sim(\mu,\nu)$,
that is, there exist $a_{2}>0$ and $b_{2}\in\mathbb{R}$ such that\[
d\mu(t)=d(\mu\boxplus_{\text{c}}\mu)(a_{2}t+b_{2})\qquad\text{and}\qquad d\nu(t)=d(\nu\boxplus\nu)(a_{2}t+b_{2}).\]
The analytic description of free convolution implies that \begin{eqnarray*}
F_{\nu}^{-1}(z) & = & \frac{1}{a_{2}}\left[F_{\nu\boxplus\nu}^{-1}(a_{2}z)-b_{2}\right]\\
 & = & \frac{2}{a_{2}}\left[F_{\nu}^{-1}(a_{2}z)-\frac{b_{2}}{2}\right]-z\\
 & = & 2F_{\nu_{2}}^{-1}(z)-z=F_{\nu_{2}\boxplus\nu_{2}}^{-1}(z),\end{eqnarray*}
where $d\nu_{2}(t)=d\nu(a_{2}t+b_{2}/2)$. This shows that $\nu=\nu_{2}\boxplus\nu_{2}$.
Moreover, Remark 5.1 and Proposition 2.2 show that \begin{eqnarray*}
\Phi_{(\mu,\nu)}(z) & = & \frac{1}{a_{2}}\left[\Phi_{(\mu\boxplus_{\text{c}}\mu,\nu\boxplus\nu)}(a_{2}z)-b_{2}\right]\\
 & = & \frac{2}{a_{2}}\left[\Phi_{(\mu,\nu)}(a_{2}z)-\frac{b_{2}}{2}\right]\\
 & = & 2\Phi_{(\mu_{2},\nu_{2})}(z)=\Phi_{(\mu_{2}\boxplus_{\text{c}}\mu_{2},\nu_{2}\boxplus\nu_{2})}(z)\end{eqnarray*}
in a truncated cone, where $d\mu_{2}(t)=d\mu(a_{2}t+b_{2}/2)$. Therefore,
we have $\mu=\mu_{2}\boxplus_{\text{c}}\mu_{2}$.

Next, we consider $(\mu_{2},\nu_{2})\sim(\mu,\nu)=(\mu_{2}\boxplus_{\text{c}}\mu_{2},\nu_{2}\boxplus\nu_{2})$.
By a slight modification of the above argument, it is easy to verify
that there exist $a_{3}>0$ and $b_{3}\in\mathbb{R}$ such that $\nu=\nu_{3}\boxplus\nu_{3}\boxplus\nu_{3}$
and $\mu=\mu_{3}\boxplus_{\text{c}}\mu_{3}\boxplus_{\text{c}}\mu_{3}$,
where $d\nu_{3}(t)=d\nu_{2}(a_{3}t+b_{3}/3)$ and $d\mu_{3}(t)=d\mu_{2}(a_{3}t+b_{3}/3)$.
Continuing in this fashion, we see that the pair $(\mu,\nu)$ is $\boxplus_{\text{c}}$-infinitely
divisible.
\end{proof}
Recall from \cite{BVunbdd} that an analytic function $\phi:\,\mathbb{C}^{+}\rightarrow\mathbb{C}^{-}\cup\mathbb{R}$
is said to be \emph{stable} if for every $a>0$, there exist $b>0$
and $c\in\mathbb{R}$ such that \[
\phi(z)+\frac{1}{a}\phi(az)=\frac{1}{b}\phi(bz)+c,\qquad z\in\mathbb{C}^{+}.\]
The next result follows immediately from Remark 5.1.

\begin{prop}
A $\boxplus_{\text{c}}$-infinitely divisible pair of measures $(\mu,\nu)$
is $\boxplus_{\text{c}}$-stable if and only if the functions $\Phi_{(\mu,\nu)}$
and $F_{\nu}^{-1}(z)-z$ are stable.
\end{prop}
A complete characterization of stable analytic functions was proved
in \cite{BVunbdd}. We will write out this result below for the sake
of completeness. The complex functions in the following list are given
by their principal value in the upper half plane.

\begin{thm}
The following is a complete list of the stable analytic functions
$\phi:\,\mathbb{C}^{+}\rightarrow\mathbb{C}^{-}\cup\mathbb{R}$.
\begin{enumerate}
\item $\phi(z)=a+ib$, $a\in\mathbb{R}$ and $b\leq0$.
\item $\phi(z)=a+bz^{-\alpha+1}$, $a\in\mathbb{R}$, $\alpha\in(1,2]$,
$b\neq0$, and $\arg b\in[(\alpha-2)\pi,0]$.
\item $\phi(z)=a+bz^{-\alpha+1}$, $a\in\mathbb{R}$, $\alpha\in(0,1)$,
$b\neq0$, and $\arg b\in[-\pi,(\alpha-1)\pi]$.
\item $\phi(z)=a+b\log z$, $a\in\mathbb{C}^{-}\cup\mathbb{R}$ and $b<0$.
\end{enumerate}
\end{thm}
Finally, we briefly outline the role of $\boxplus_{\text{c}}$-stable
pairs of measures in relation to the limit theorems. Following the
ideas in \cite{PataLevy}, one can show that a pair of measures $(\mu,\nu)$
is $\boxplus_{\text{c}}$-stable if and only if there exist $A_{n}>0$,
$B_{n}\in\mathbb{R}$ and measures $\mu^{\prime},\nu^{\prime}\in\mathcal{M}$
so that the measure $\mu$ (resp., $\nu$) is the weak limit of c-free
(resp., free) convolutions $\underbrace{\mu_{n}\boxplus_{\text{c}}\mu_{n}\boxplus_{\text{c}}\cdots\boxplus_{\text{c}}\mu_{n}}_{n\,\text{times}}$
(resp., $\underbrace{\nu_{n}\boxplus\nu_{n}\boxplus\cdots\boxplus\nu_{n}}_{n\,\text{times}}$),
where the measure $\mu_{n}$ and $\nu_{n}$ are given by \[
d\mu_{n}(t)=d\mu^{\prime}(A_{n}t+B_{n}),\qquad\text{and}\qquad d\nu_{n}(t)=d\nu^{\prime}(A_{n}t+B_{n}).\]
We will not provide the details of the proof of the above assertion
because it is quite similar to those in the free case \cite{PataLevy}.
The reader will have no difficulty in providing his/her own proof.

\subsubsection*{Acknowledgments}

The author is grateful to his adviser Professor Hari Bercovici for
his guidance and encouragement throughout this work. He also benefited
a lot from various conversations with Professors Serban Belinschi
and Richard Bradley.

\end{document}